\documentclass[12pt]{article}
\usepackage{amsmath,amssymb,latexsym}
\usepackage{graphicx,psfrag,epsf}
\usepackage{enumerate}
\usepackage[numbers]{natbib}
\usepackage{wrapfig}
\usepackage{xcolor}
\usepackage{parskip}
\usepackage{physics}
\usepackage{amsthm}
\usepackage{etoolbox}
\usepackage[title,titletoc]{appendix}
\usepackage{titlesec}
\newcommand{\blind}{0}

\newtheorem{theorem}{Theorem}
\newtheorem{corollary}{Corollary}
\newtheorem{lemma}{Lemma}
\newtheorem{prop}{Proposition}
\newtheorem{definition}{Definition}
\begin{document}

\def\spacingset#1{\renewcommand{\baselinestretch}%
{#1}\small\normalsize} \spacingset{1}


\if0\blind
{
  \title{\normalsize \textbf{HAUSDORFF MEASURE BOUND FOR THE NODAL SETS OF NEUMANN LAPLACE EIGENFUNCTIONS}}
  \author{\normalsize SHAGHAYEGH FAZLIANI}
  \date{}
  \maketitle

} \fi

\if1\blind
{
  \bigskip
  \bigskip
  \bigskip
  \begin{center}
    {\LARGE\bf Title}
\end{center}
  \medskip
} \fi

\bigskip
\begin{abstract}

We study the nodal sets of Neumann Laplace eigenfunctions in a bounded domain with $\mathcal{C}^{1,1}$ boundary.  We show that for  $u_\lambda$ such that $\Delta u_\lambda + \lambda u_\lambda = 0 $ with the Neumann boundary condition $\partial_\nu u_\lambda = 0$, we have $\mathcal{H}^{n-1}(\{u_\lambda = 0\}) \leq C \sqrt{\lambda}$.
\end{abstract}

\spacingset{1.45}


\section{Introduction}
 For the eigenfunctions of the Laplacian
 \begin{equation*}
     \Delta u + \lambda u = 0
 \end{equation*}
on a compact smooth $n$ dimensional Riemannian manifold $M$, Yau \cite{Yau1982} conjectured that 
the $n-1$ dimensional Hausdorff measure of the set of zeros of $u_\lambda$, also referred to as the nodal set of $u_\lambda$, can be controlled by the following upper and lower bounds
\begin{equation}
    c\sqrt{\lambda} \leq \mathcal{H}^{n-1}(\{u_\lambda = 0\} \cap \Omega) \leq C\sqrt{\lambda},
\end{equation}
where $c$ and $C$ only depend on the manifold and are independent of the eigenvalue $\lambda$.  This conjecture was proved by Donnelly and Fefferman \cite{MR0943927} for real analytic metrics. Logunov obtained the lower bound in \cite{MR3739232} and the polynomial upper bound for compact $\mathcal{C}^\infty$-smooth Riemannian manifolds in \cite{MR3739231}. Later, Logunov, Malinnikova, Nadirashvili, and Nazarov \cite{MR4356702} obtained the sharp upper bound for the Dirichlet Laplace eigenfunctions of a domain with $\mathcal{C}^1$ boundary.  
Recently, J. Zhu and J. Zhuge obtained the sharp upper bound for the Dirichlet Laplace case in quasiconvex domains in \cite{zhu2023nodal}.

In this paper, we focus on the nodal sets of the Neumann Laplace eigenfunctions. For a bounded domain $\Omega \subset \mathbb{R}^n$, a Neumann Laplace eigenfunction $u_\lambda$ is a solution of
\begin{equation}\label{NeumannLaplaceEigenfunction}
    \Bigg\{
    \begin{array}{ll}
         \Delta u + \lambda u = 0 & \text{\quad in } \Omega, \\
       \partial_\nu u = 0 & \text{\quad on } \partial \Omega,
    \end{array}
\end{equation}

where we denote $\partial_\nu$ to be the outward normal derivative. The Laplacian operator $-\Delta$ is non-negative. However, unlike the case of the Dirichlet boundary condition, $\Omega$ solely being bounded is not sufficient to result in a discrete spectrum for the Neumann Laplacian. Regarding this result, it has been proven in \cite{MR2251558}, Theorem 1.2.8, that when $\Omega$ is bounded and Lipschitz, the Neumann Laplacian has a discrete spectrum of infinitely many non-negative eigenvalues with no finite accumulation point
\begin{equation*}
    0 = \lambda_1(\Omega) \leq \lambda_2(\Omega) \leq \lambda_3(\Omega) \leq ...\quad,
\end{equation*}
which mainly results from the compactness of the embedding $H^1(\Omega) \hookrightarrow {L}^2(\Omega)$. In this paper, we will prove the following main result on the Hausdorff measure of the nodal sets of the Neumann Laplace eigenfunctions in a bounded $\mathcal{C}^{1,1}$ domain.  
\begin{theorem}\label{MainTheorem}
    Let $\Omega$ be a bounded domain in $\mathbb{R}^n$ with $\mathcal{C}^{1,1}$ boundary, and let $u_\lambda$ be a solution of \eqref{NeumannLaplaceEigenfunction} in $\Omega$. Denote by $Z_{u_\lambda} := \{u_\lambda = 0\}$ the zero set of $u_\lambda$. Then, there exists a constant $C > 0$, depending only on the domain, such that 
\begin{equation}\label{HausdorffMeasureBound}
\mathcal{H}^{n-1}(Z_{u_\lambda} \cap \Omega) \leq C \sqrt{\lambda}.
\end{equation}
\end{theorem} 

J. Zhu proved in \cite{MR4594557} that the upper bound \eqref{HausdorffMeasureBound} holds in domains with real analytic boundaries. In this paper, we take the boundary regularity needed to prove this result down to $\mathcal{C}^{1,1}$. In order to estimate the size of the nodal set of the Neumann eigenfunction, we first multiply the eigenfunction by an exponential function of a new variable and reduce the problem to estimates of the zero sets of harmonic functions with the Neumann boundary condition and bounded frequency. For such functions, locally near the boundary, one may consider a reflection of a solution of a divergence form elliptic equation with Lipschitz coefficients. The idea of a reflection across the boundary is well-known for the case of smooth boundaries; see for instance \cite{MR3396451}, where it was originally done, and also \cite{decio2022hausdorff} and \cite{MR4577966}, which later used a similar idea. In this article, we consider a smooth version of the reflection which can be done on $\mathcal{C}^{1,1}$ domains. The author is grateful to Zihui Zhao for the idea of the construction and the details of it. A similar construction and variants of it are used in a paper by C. E. Kenig and Zihui Zhao, see \cite{kenig2024note}. The construction of the reflection in this paper is mentioned in Section 2 and developed in detail in Appendix A. Using this construction, we go on to define the appropriate frequency function and doubling index and prove related results in Section 3, which pave the way for proving our auxiliary lemmas in Section 4. Finally, we dedicate Section 5 to the proof of Theorem 1 using mainly the results in Section 4. We also provide Appendix B dedicated to examples of the nodal sets of the Neumann eigenfunctions on some special domains, rectangles and disks, and how the constant in equation \eqref{HausdorffMeasureBound} depends on the domains in these cases.

\section{Extension and Reflection Across the Boundary}
\subsection*{Harmonic Extension}
Let $\Omega_0 \subset \mathbb{R}^n$ be a bounded domain and $u_\lambda$ a solution of \eqref{NeumannLaplaceEigenfunction} in $\Omega_0$. We consider the harmonic extension of $u_\lambda$ to the domain $\Omega = \Omega_0 \times \mathbb{R} \subset \mathbb{R}^{n+1}$ to be 
\begin{equation*}
  h(x,t) = u_\lambda(x) e^{\sqrt{\lambda}t}.  
\end{equation*}
Then, $h \in \mathcal{C}(\Bar{\Omega})$, and clearly, $Z_h = Z_{u_\lambda }\times \mathbb{R}$, where the zero sets are sets inside the domains $\Omega$ and $\Omega_0$ respectively. Now we have
\begin{equation}\label{NeumannLaplaceHarmonic}
    \Bigg\{
    \begin{array}{ll}
         \Delta h = 0 & \text{\quad in } \Omega, \\
       \partial_\nu h = 0 & \text{\quad on } \partial \Omega.
    \end{array}
\end{equation}
We will utilize the function $h$, which we refer to as the harmonic extension of $u_\lambda$, to prove the results required to get to Theorem \ref{MainTheorem}. 
\subsection*{Reflection Across the Boundary}
The next step forward is to introduce a local reflection of $h$ across the boundary of $\Omega$. The details of the construction of this reflection and related calculations can be found in Appendix A. There, it  is explained why the $\mathcal{C}^{1,1}$ boundary is required in this paper. To formulate this more formally, let us proceed with the following lemma.
\begin{lemma} \label{BoundaryBallCovering}
    \textnormal{\textbf{(Boundary Covering  with Balls)}} There exists $r_0 > 0$ depending on the domain  such that we can cover $\partial \Omega$ by a set $\bigcup_\beta B(x_\beta, r_0)$ consisting of a finite number of balls of radius $r_0$ centered on $\partial \Omega$ such that for every $\beta$, a reflection $\widetilde{h}_\beta$ of $h$ exists in $B(x_\beta, r_0)$, which satisfies a uniformly elliptic equation div$(B(x) \grad \widetilde{h}_\beta) = 0$ in $B(x_\beta, r_0)$ (Note that $B(x) = I$ inside the domain). Moreover, the matrix $B$ is symmetric with Lipschitz coefficients and is only dependent on $\partial \Omega \cap B(x_\beta, r_0)$.
\end{lemma}
\begin{proof}
    The proof follows directly from the results in Appendix A. By the construction there, we can choose $r_0$ small enough such that for any $x_0 \in \partial \Omega$, a reflection $\widetilde{h}_\beta$ of $h$ exists in $B(x_0, r_0)$ such that div$(B(x) \grad \widetilde{h}_\beta) = 0$. Lemma \ref{BLipschitzCoeffs} guarantees that this uniformly elliptic equation has Lipschitz coefficients with Lipschitz constant only depending on $\Omega$. Note that the matrix $B$ is symmetric and independent of $h$ due to the way it is constructed, but it is dependent on $\partial \Omega \cap B(x_0, r_0)$.
\end{proof}

 We fix $r_0$ and the covering with balls $\partial{\Omega} \subset \bigcup_\beta B(x_\beta, r_0)$ for the rest of the paper and refer to it later in the next constructions.

\section{Frequency Function and Doubling Index}

Let $v \in W_{loc}^{1,2}(B)$ be a solution to \begin{equation} \label{DivergenceForm}
    \mathcal{L}v = \text{div}(A\grad v) = 0
\end{equation} 
in $B = B(0,1) \in \mathbb{R}^k$, where $A$ is a symmetric uniformly positive definite Lipschitz matrix and $A(0) = I$. Let further 
\begin{equation*}
    \alpha(x) = \frac{(A(x)x,x)}{|x|^2}, \quad \alpha(0) = 1
\end{equation*}
Since $A$ has Lipschitz coefficients, we have $A(x) = I + O(|x|)$ and $\alpha(x) = 1 + O(|x|)$. Also, for $B(x,r) \subset B$ let $H_v(x,r)$ be defined as
\begin{equation*}
    H_v(x,r) = \int_{B(x,r)} \alpha v^2
\end{equation*}
\begin{definition} \label{FrequencyFunction}
Take $H_v'(x,r)$ to be the derivative of $H_v(x,r)$ with respect to $r$. For $B(x,r) \subset B$, the frequency function ${N}_v(x,r)$ is defined by 
\begin{equation} \label{FrequencyFunctionVer1Eq}
   N_v(x,r) = \frac{rH_v'(x,r)}{H_v(x,r)}.
\end{equation}
\end{definition} 

\begin{definition} \label{DoublingIndexVer1}
    For $B(x,2r) \subset B$, the doubling index $\mathcal{N}_v(x,r)$ is defined by 
\begin{equation} \label{DoublingIndexVer1Eq}
    \mathcal{N}_v(x,r) = \log \frac{ H_v(x,2r)}{ H_v(x,r)}.
\end{equation}
\end{definition}

For the frequency function and doubling index inside of the domain, the following ``comparability" and ``almost monotonicity" properties have been proven in multiple sources, including \cite{MR0833393}, \cite{garofalo1987unique}, and \cite{MR4249624}. We will briefly state these results and refer the reader to the aforementioned references for the proofs.

\begin{prop} \label{FrequencyFucntionAlmostMonotonicity} Let $v$ be a solution to \eqref{DivergenceForm}. Then, there exist constants $\rho_0, C_1$ and $C_2$ depending on the Lipschitz coefficients and the ellipticity properties of the elliptic equation \eqref{DivergenceForm} such that for any $0 < 2 r_1 < r_2 < \rho_0$ 
\begin{equation*}
    N_v(x,r_1) \leq C_1 + C_2 N_v(x,r_2)
\end{equation*}
\end{prop}

\begin{prop} \label{Comparability}
 Let $v$ be a solution to \eqref{DivergenceForm}. Then, there exist constants $\rho_0, C_1$ and $C_2$ depending on the Lipschitz coefficients and the ellipticity properties of the elliptic equation \eqref{DivergenceForm} such that for any $0 < 4r < \rho_0$ 
\begin{equation*}
    C_1 ^{-1} N_v(x,r) - C_2 \leq \mathcal{N}_v(x,r) \leq  C_1 N_v(x,2r) + C_2
\end{equation*}
\end{prop}

\begin{prop} \label{DoublingIndexAlmostMonotonicity}
 Let $v$ be a solution to \eqref{DivergenceForm}. Then, there exist constants $\rho_0$, $C_1$, and  $C_2$ depending on the Lipschitz coefficients and the ellipticity properties of the elliptic equation \eqref{DivergenceForm} such that for any $0 < 2 r_1 < r_2 < \rho_0$ 
\begin{equation*}
    (\frac{r_2}{r_1})^{C_1^{-1}\mathcal{N}_v(x,r_1)-C_2} \leq \frac{H_v(x,r_2)}{H_v(x,r_1)} \leq   (\frac{r_2}{r_1})^{C_1\mathcal{N}_v(x,r_2)+C_2}.
\end{equation*}
\end{prop}

We would like to define a version of the frequency function (and the doubling index in a similar manner) for our harmonic function $h$ which takes into account the boundary and has the ``almost monotonicity" property. To do this, we will introduce a partition of the domain $\Omega$ into cubes.

\subsection*{Partition of the domain into cubes}
Using the notation of Lemma \ref{BoundaryBallCovering}, we cover $\partial \Omega$ with the finite set  of balls $\bigcup_\beta B(x_\beta, r_0)$, and we have a reflection $\widetilde{h}_\beta$ of $h$ inside each ball of $\bigcup_\beta B(x_\beta, r_0)$. Now, we will construct a partition of the domain into cubes in the following lemma. 

\begin{lemma}\label{CubeDecomposition}
    \textnormal{\textbf{(Decomposition into Cubes)}}
    There exists a finite covering of the closure of the domain $\bar{\Omega} \subset \bigcup_j Q_j^{bdry} \cup \bigcup_i Q_i^{int}$  with cubes such that the following properties hold.
    \item a) Cubes $\{Q_j^{bdry}\}$ are disjoint of small enough side length $s$ centered on $\partial \Omega$. Moreover, in each $Q_j^{bdry}$, we have a reflection $\widetilde{h}_j$ of $h$ across the boundary of $\Omega \cap Q_j^{bdry}$. 
    \item b) Cubes $Q_i^{int}$ are chosen such that for some small constant $c > 0$, the distance from every $Q_i^{int}$ to the boundary is at least  $c\cdot s(Q_i^{int})$, where $s(Q_i^{int})$ is the side length of the cube.
    \end{lemma}
\begin{proof}
    we first cover the boundary of $\Omega$ with a finite number of disjoint cubes $\{Q_j^{bdry}\}$ of small enough side length $s$ centered at points on $\partial \Omega$ such that every  $Q_j^{bdry}$ fits inside some ball of the covering $\bigcup_\beta B(x_\beta, r_0)$. By Lemma \ref{BoundaryBallCovering}, we have a reflection $\widetilde{h}_\beta$ of $h$ across the boundary of $\Omega \cap B(x_\beta, r_0)$. Now, if $Q_j^{bdry} \subset B(x_\beta, r_0)$, $\widetilde{h}_j = \widetilde{h}_\beta|_{Q_j^{bdry}}$ would be the reflection of $h$ across the boundary of $\Omega \cap Q_j^{bdry}$. This concludes property $a$. To have $b$,  just decompose $\Omega\backslash\bigcup_j Q_j^{bdry}$ into cubes with disjoint interiors denoted as $\{Q_i^{int}\}$ with the stated property. Note that in this construction, some cubes from  $\{Q_i^{int}\}$ might intersect with some boundary cubes $Q_j^{bdry}$. 
\end{proof}
We fix this decomposition into cubes $\bar{\Omega} \subset \bigcup_j Q_j^{bdry} \cup \bigcup_i Q_i^{int}$ for the rest of the paper. Now, let us get back to the discussion of defining a doubling index for $h$ that takes into account the boundary. Suppose $h \in \mathcal{C}(\Bar{\Omega})$ satisfies equation \eqref{NeumannLaplaceHarmonic} in $\Omega$ with $\mathcal{C}^{1,1}$ boundary, and take $\widetilde{h}_j$ to be the reflection across the boundary to $h$ in each of the cubes $Q_j^{bdry}$. Note by this construction, for every $j$, there is a different reflection $\widetilde{h}_j$ of $h$ in $\Omega \cup Q_j^{bdry}$.

\begin{definition} \label{DefVer2}
For $x \in Q_j^{bdry}$, $h \in \mathcal{C}(\Bar{\Omega})$ harmonic, and $B(x,2r) \subset Q_j^{bdry}$, the doubling index 
of $h$ close to the boundary, denoted by $\mathcal{N}_h^*(x,r)$, is defined by 
\begin{equation} \label{DefVer2Eq}
   \mathcal{N}_h^*(x,r) =  \mathcal{N}_{\widetilde{h}_j}(x,r),
\end{equation}
where $\mathcal{N}_{\widetilde{h}_j}(x,r)$ is defined as in equation \eqref{DoublingIndexVer1Eq} and $ H_{\widetilde{h}_j}(x,r) =  \int_{B(x,r) \cap \Omega} h^2 + \int_{B(x,r)\backslash \Omega} \alpha \widetilde{h}_j^2$. 
\end{definition}

Clearly, equation \eqref{DefVer2Eq} in Definition \ref{DefVer2} takes the doubling index of $h$ in $B(x,r)\cap \Omega$ to be equal to the classical definition of the doubling index of $\widetilde{h}_j$ inside of the extended domain in $B(x,r)$. Since we have the ``almost monotonicity" property of the doubling index for $\widetilde{h}_j$ as the solution to the elliptic equation div$(B(x) \grad u) = 0$ with Lipschitz coefficients and symmetric uniformly elliptic $B$, $ \mathcal{N}_h^*(x,r)$ is also almost monotonic. We sum up this property in the following Lemma.
\begin{lemma} \label{AlmostMonotonicity}
 The doubling index $\mathcal{N}_h^*(x,r)$ has the almost monotonicity property in the sense of Proposition $\ref{DoublingIndexAlmostMonotonicity}$, which results in the existence of constants $\rho_0$, $C$, and $c$ depending on the elliptic equation such that for any $0 < 2r_1 < r_2 < \rho_0$,
\begin{equation*}
\mathcal{N}_h(x,r_1) \leq C \mathcal{N}_h(x,r_2)+ c.
\end{equation*}
    \end{lemma}
\begin{proof}
    Straightforward  from Definition \ref{DefVer2} and Proposition \ref{DoublingIndexAlmostMonotonicity}.
\end{proof}

\section{Two Important Lemmas}
In this Section, we  prove two important lemmas regarding estimates for the doubling index, Lemma \ref{SmallCubeHalf} and Lemma \ref{sqrtBoundforN}, using the decomposition of our domain into cubes constructed in the previous section. Let us begin with the following definition for the doubling index of $h$ in a cube.

\begin{definition} \label{CubeDoublingIndex}
Take  $Q$ to be a cube centered in $\Bar{\Omega}$ contained in one of the boundary cubes $Q_j^{bdry}$ or one of the interior cubes $Q_i^{int}$ in the decomposition $\bar{\Omega} \subset \bigcup_j Q_j^{bdry} \cup \bigcup_i Q_i^{int}$ we fixed from Lemma $\ref{CubeDecomposition}$.
Then, let the doubling index of $h$ in $Q$ be defined by 
\begin{equation*}
    \mathcal{N}_h^{*}(Q) = \sup_{x \in Q \cap \Omega, 0 < r \leq \text{diam}(Q)} \mathcal{N}_h^*(x,r).
\end{equation*}
\end{definition} 
Note that since $(x,r) \rightarrow \mathcal{N}_h^*(x,r)$ is a continuous functions on $Q \times [0, \text{diam}(Q)]$, the above supremum is finite.

\begin{lemma}\label{SmallCubeHalf}
    Let $Q$ be one of the boundary cubes $\{Q_j^{bdry}\}$ of side length $s$  in the decomposition $\bigcup_j Q_j^{bdry} \cup \bigcup_i Q_i^{int}$ of $\bar{\Omega} \subset \mathbb{R}^k$ into cubes. Then, there exists $M$ and $N$ large enough such that for any $h$ satisfying \eqref{NeumannLaplaceHarmonic} in $\Omega$,  if $Q \cap \partial \Omega$ is covered by $2^{M}$ cubes $\{q_j\}$ of side length $2^{-M}s$ centered on $\partial\Omega$ with disjoint interiors and $\mathcal{N}_0 := \mathcal{N}_h^{*}(Q) > N$,  there exists a cube $q \in \{q_j\}$ such that $\mathcal{N}_h^{*}(q) \leq \frac{\mathcal{N}_h^{*}(Q)}{2}$.
    \end{lemma} 

To prove this lemma, we will use the following proposition which is a result of Theorem 1.7 in \cite{Alessandrini_2009}.
\begin{prop} \label{CauchyData}
Let $D \in \mathbb{R}^k$ be a bounded domain with Lipschitz boundary, and take $B$ to be a ball of radius $R$ centered in $\Bar{D}$. Let $v \in W_{loc}^{1,2}(B)$ be a solution to $ \mathcal{L}v = \text{div}(A\grad v) = 0 $ in $B \cap D$, where $A$ is a uniformly positive definite Lipschitz matrix. Suppose for some $\epsilon > 0$ we have $||v||_{L^2(\frac{1}{2}B \cap \partial D)} \leq \epsilon$ and $||\grad v||_{L^2(\frac{1}{2}B \cap \partial D)} \leq \frac{\epsilon}{R}$. Then, there exists $0 < \beta < 1$ and $C > 0$ depending on $D$ and the Lipschitz and ellipticity constants of $A$ such that 
\begin{equation*}
    ||v||_{L^2(\frac{1}{2}B \cap D)} \leq C \epsilon^\beta ||v||^{1-\beta}_{L^2(B \cap D)}.
\end{equation*}
\end{prop}

      \textit{Proof of Lemma $\ref{SmallCubeHalf}$}. Let $\widetilde{h}$ be the reflection of $h$ across the boundary in $Q$ from Lemma \ref{BoundaryBallCovering}. Let $x_Q \in \partial \Omega$ be the center of the cube $Q$. Take $B_M(Q) := \{q_j\} $ to be the set of all the cubes $q_j$  of side length $2^{-M}s$ covering $Q$ as described in the statement of Lemma \ref{SmallCubeHalf}. Suppose now that the inequality $\mathcal{N}_h^{*}(q) > \frac{\mathcal{N}_h^{*}(Q)}{2} = \frac{\mathcal{N}_0}{2}$ holds for all $q \in B_M(Q)$ for a fixed large enough $M$. Then, based on the definition of $\mathcal{N}_h^{*}(q)$, for each of these cubes $q \in B_M(Q)$, there exist $z_q \in q$ and $r_q > 0$ satisfying $ r_q \leq \text{diam}(q) = 2^{-M}s\sqrt{k}$ such that $\mathcal{N}_h^{*}(z_q, r_q) > \frac{\mathcal{N}_0}{2}$. Set $R := s\sqrt{k}$ and $r_M := 2^{-M}s\sqrt{k}$. Using the almost monotonicity property for the doubling index in Lemma \ref{AlmostMonotonicity}, assuming $\mathcal{N}_0$ is large enough, we have
      \begin{equation*}
          H_{\widetilde{h}}(z_q, 2r_M) \leq 2^{-c(M-1)\mathcal{N}_0}  H_{\widetilde{h}}(z_q, \frac{R}{2}) \leq e^{-\widetilde{c}M\mathcal{N}_0}  H_{\widetilde{h}}(x_Q, R)
      \end{equation*}
      For constants $c$ and $\widetilde{c}$ and large enough $M$.  Consider the ball $B := B(x_Q, R)$  and set 
      $H := H_{\widetilde{h}}(x_Q, R) = H_{\widetilde{h}}(B)$, then we have  $2Q \subset B$ and
         \begin{equation} \label{eq1}
          H_{\widetilde{h}}(z_q, 2r_M) \leq e^{-\widetilde{c}M\mathcal{N}_0}  H.
      \end{equation}
     The reflection $\widetilde{h}$ satisfying an elliptic equation with Lipschitz coefficients gives us access to elliptic estimates such as $||\grad \widetilde{h}||_{L^2(B(x,r))} \leq \frac{C}{r}
     ||\widetilde{h}||_{L^2(B(x,2r))}$, which is a direct result of Caccioppoli's inequality.  Using this gradient estimate and \eqref{eq1}, we get
     \begin{equation} \label{eq2}
          ||\grad \widetilde{h}||_{L^2(B(z_q,r_M))} \leq  \frac{C}{r_M} e^{-\widetilde{c}M\mathcal{N}_0}  H.
      \end{equation}
      Covering $\frac{1}{2}B \cap \partial\Omega$ with $2^M$ balls of the form $\{B(z_q, r_M)\}$, \eqref{eq1} and \eqref{eq2} give us bounds for the Cauchy data of $h$ on $\frac{1}{2}B \cap \partial\Omega$. Now, Proposition \ref{CauchyData} gives us 
      \begin{equation*}
          ||\widetilde{h}||_{L^2(\frac{1}{2}B)} \leq C 2^{\beta M} e^{-\widetilde{c} \beta M\mathcal{N}_0}  H,
      \end{equation*}
    which, for large enough $M$, results in
    \begin{equation*}
          H_{\widetilde{h}}(\frac{1}{2}B) \leq \widetilde{C} e^{\beta M -\widetilde{c} \beta M\mathcal{N}_0}  H.
      \end{equation*}
      Now we have
\begin{equation*}
    \mathcal{N}_h^{*}(x_Q, \frac{1}{2} s \sqrt{k})  = \log \frac{H_{\widetilde{h}}(x_Q, s\sqrt{k})}{H_{\widetilde{h}}(x_Q, \frac{1}{2} s\sqrt{k})} = \log \frac{H_{\widetilde{h}}(B)}{H_{\widetilde{h}}(\frac{1}{2}B)} \geq \log \frac{H}{\widetilde{C} e^{\beta M - \widetilde{c}\beta\mathcal{N}_0M} H} \geq C(\widetilde{c}\beta\mathcal{N}_0M - \beta M ).
\end{equation*}
Taking $M$ large enough, the rightmost term would grow larger than $\mathcal{N}_0$, which is a contradiction to the assumption that ${\mathcal{N}_h^{*}(Q)} = {\mathcal{N}_0}$. Note that constants $C, \widetilde{C}, c, \text{ and } \widetilde{c}$ in the above proof vary in every line, but they are all constants depending on our domain and the elliptic equation. \qed \\
To prove our second important lemma, we use the following proposition from \cite{Alessandrini_2009} (Theorem 2.1), which is the \textbf{three-spheres inequality} for solutions to elliptic equations. 
\begin{prop} \label{3ball}
   \textnormal{\textbf{(Three-spheres inequality)}} Let $v \in W_{loc}^{1,2}(B)$ be a solution to $\text{div}(A\grad v) = 0 $ in $B_R = B(0,R) \in \mathbb{R}^k$, where $A$ is a uniformly positive definite Lipschitz matrix. Then, for every $r_1, r_2$, and $r_3$, with $0 < r_1 < r_2 < r_3 < R$,
   \begin{equation*}
       ||v||_{L^2(B_{r_2})} \leq C ||v||_{L^2(B_{r_1})}^\beta ||v||_{L^2(B_{r_3})}^{1-\beta},
   \end{equation*}
   where $C > 0$ and $ 0 < \beta < 1$ only depend on $R$ and the Lipschitz properties of the elliptic equation and the ellipticity constant as well as $r_1, r_2$, and $r_3$.
\end{prop}
\begin{lemma}\label{sqrtBoundforN}
   Let $\Omega_0$ be a bounded $\mathcal{C}^{1,1}$ domain in $\mathbb{R}^n$ and set $\Omega := \Omega_0 \times [-1,1] \subset \mathbb{R}^{n+1}$. Then, for any $r$ small enough, there exists $C > 0$ depending on $\Omega_0$ such
that for any Neumann Laplace eigenfunction $u_\lambda$, the corresponding harmonic
extension $h$ of $u_\lambda$ satisfies $\mathcal{N}_h^{*}(y, r) \leq C\sqrt{\lambda}$ for every
$y = (x, t) \in \bar{\Omega}$.
\end{lemma} 
    \begin{proof}
    Take $y = (x,t) \in \bar{\Omega}_0 \times \{t\} \subset \bar{\Omega}$. From the definition of $\mathcal{N}_h^{*}(y,r)$, we know that $\mathcal{N}_h^{*}(y,r) = \mathcal{N}_{\widetilde{h}}(y,r)$ in a neighborhood of $y$ where we have a reflection, therefore, it suffices to show that $\mathcal{N}_{\widetilde{h}}(y,r) \leq C\sqrt{\lambda}$ for every
$y = (x, t) \in \bar{\Omega}$ for an appropriate reflection $\widetilde{h}$. Note that since we have monotonicity for $\mathcal{N}_h^*(y,r)$ by Lemma \ref{AlmostMonotonicity}, it suffices to show the inequality for a fixed $r = \rho$, then we would have the result for all $r \leq \rho$. For convenience and without loss of generality, take $\rho = 1$. Take a finite set $S \subset \Bar{\Omega}$ such that $\Bar{\Omega}_0 \subset \bigcup_{z \in S} B(z, \frac{1}{8})$. Set $B = B(y,1) \subset \mathbb{R}^{n+1}$ and suppose the maximum of $u_{\lambda}$ in $\Bar{\Omega}_0$ takes place at a point $x_0 \in \bar{\Omega}$ and that $\max_{\Bar{\Omega}_0} |u_{\lambda}| = |u_{\lambda}(x_0)| = 1$. Take a path $\gamma: [0,1] \rightarrow \Bar{\Omega}_0$ from $x$ to $x_0$ such that $\gamma((0,1)) \subset \Omega_0$. We construct a chain of balls of radius $\frac{1}{2}$ like $\{B_{\ell}\}_{\ell = 0}^{L} \subset \mathbb{R}^n$ as follows. Set $B_0 = B(x,\frac{1}{2}) $ and suppose we have constructed all the balls $B_j = B(x_j, \frac{1}{2})$ for $1 \leq j \leq \ell$ as well. Then, to construct $B_{\ell + 1}$ we define 
\begin{equation*}
    s_\ell = \sup  \{ s \in [0,1]: |\gamma(s) - x_\ell | \leq \frac{1}{8} \}
\end{equation*}
If $s_\ell = 1$, we set $L = \ell + 1$ and  $x_{\ell + 1} = x_0$, and the chain stops. However, if $s_\ell < 1$, we must have $|\gamma(s_\ell) - x_\ell | = \frac{1}{8}$, so we can choose $x_{\ell + 1} \in S\backslash \{x, x_1, ..., x_\ell \}$ such that $|x_{\ell + 1} - \gamma(s_\ell)| < \frac{1}{8}$, which results in $|x_{\ell + 1} - x_\ell| < \frac{1}{4}$ as well. Define $B_{\ell + 1} = B(x_{\ell + 1},\frac{1}{2})$. Constructed in this way, we have $B_{\ell + 1} \subset \frac{3}{2} B_\ell$ for every $ 0 \leq \ell \leq L$. The constructed chain is finite, and the number of balls in the chain is bounded by $|S| + 2$. Now, set $y_\ell =  (x_\ell,t)$ and $\widetilde{B}_\ell = B(y_\ell, \frac{1}{2}) \subset \mathbb{R}^{n+1}$. We can assume without loss of generality that  $4\widetilde{B}_\ell \subset Q \cup \Omega$, where $Q \in \{Q_i^{int}\}$ is a boundary cube in the decomposition $\bar{\Omega} \subset \bigcup_j Q_j^{bdry} \cup \bigcup_i Q_i^{int}$ we fixed in Lemma \ref{CubeDecomposition} and $y_0 = (x_0, t) \in Q$. By Lemma \ref{BoundaryBallCovering}, take $\widetilde{h}$ to be the reflection of $h$ across the boundary in $Q$. Using the notation of Lemma \ref{Phi}, we suppose $\widetilde{h}(X) = h(\Phi(X))$ for any $X \in U$. Since $y_0 \in \Bar{\Omega}$, we have $\Phi(y_0) = y_0$. Now, Corollary \ref{DistanceEstimatePhi} tells us that there exists $c_1 > 0$ such that for every $Y \in U$
\begin{equation*}
    |\Phi(Y) - y_0| \leq c_1 |Y - y_0|.
\end{equation*}
Therefore,
\begin{equation*}
    \sup_{4\widetilde{B}_{\ell}} |\widetilde{h}| = \sup_{4\widetilde{B}_{\ell}} |h(\Phi)| \leq 
    \sup_{B(y_\ell, 2c)} |h| \leq e^{2c_1\sqrt{\lambda}}
\end{equation*}
Since  $4\widetilde{B}_\ell \subset U \cup \Omega$, by the Three-spheres inequality in Proposition \ref{3ball},
\begin{equation*}
    ||\widetilde{h}||_{L^2(\frac{3}{2}\widetilde{B}_\ell)} \leq C_1 ||\widetilde{h}||_{L^2(\widetilde{B}_\ell)}^{\frac{1}{2}} ||\widetilde{h}||_{L^2(4\widetilde{B}_\ell)}^{\frac{1}{2}} \leq C_1 e^{c_1\sqrt{\lambda}}||\widetilde{h}||_{L^2(\widetilde{B}_\ell)}^{\frac{1}{2}}.
\end{equation*}
Thus, we have
\begin{equation*}
||\widetilde{h}||_{L^2(\widetilde{B}_\ell)} \geq C_1^{-1} e^{-2c_1\sqrt{\lambda}}||\widetilde{h}||_{L^2(\frac{3}{2}\widetilde{B}_\ell)}^2   \geq C_1^{-1} e^{-2c_1\sqrt{\lambda}}||\widetilde{h}||_{L^2(\widetilde{B}_{\ell +1})}^2 \geq C_1^{-1} e^{-2c_1\sqrt{\lambda}}||h||_{L^2(\widetilde{B}_{\ell +1} \cap \Omega)}^2. 
\end{equation*}
Now, note that $\sup_{\widetilde{B}_L \cap \Omega} |h| = e^{\frac{\sqrt{\lambda}}{2}}$. Combining these inequalities and using the mean-value theorem for the subharmonic function $h^2$ in $\widetilde{B}_L \cap \Omega$,  we get $C_2 > 0$ and $c_2 > 0$ such that 
\begin{equation*}
    \sup_{\widetilde{B}_{0}} |\widetilde{h}| \geq \frac{1}{\sqrt{|\widetilde{B}_0|}}|\widetilde{h}||_{L^2(\widetilde{B}_0)} \geq C_2 e^{-c_2\sqrt{\lambda}}
\end{equation*}

Finally, using the above inequalities, for $r \leq 1$, we get 
\begin{equation*}
\begin{split}
  \mathcal{N}_h^{*}(y,r) \leq C' \mathcal{N}_h^{*}(y,1) + C'' =  C' \log \frac{H_{\widetilde{h}}(4\widetilde{B}_0)}{H_{\widetilde{h}}(2\widetilde{B}_0)} + C''  & \leq C' \log \frac{ \sup_{4\widetilde{B}_{0}} |\widetilde{h}|^2}{ \sup_{\widetilde{B}_{0}} |\widetilde{h}|^2} + C_3 \\  & \leq C'(4c_1 + 2c_2) \sqrt{\lambda} + C_4 \leq C\sqrt{\lambda}.
\end{split}
\end{equation*}
Note that we found this constant for a fixed extension $\widetilde{h}$, but since we have finitely many extensions in our finite partition of the boundary of $\Omega$, by taking the maximum of such constants, we get a global constant $C>0$ which satisfies our desired result. Clearly, $C > 0$ here only depends on the initial domain $\Omega_0$.
\end{proof}

\section{Proof of the Main Theorem}

We need the following important theorem. 

\begin{theorem}\label{Theorem2}
    Let $h$ be a harmonic function satisfying \eqref{NeumannLaplaceHarmonic} in a domain $\Omega \subset \mathbb{R}^k$ with $\mathcal{C}^{1,1}$ boundary. There exists $C>0$ independent of $h$ such that for every  cube $Q$ in the decompposition $\bar{\Omega} \subset \bigcup_j Q_j^{bdry} \cup \bigcup_i Q_i^{int}$, we have
    \begin{equation}\label{ImportantTheoremIneqality}
        \mathcal{H}^{k-1}(Z_{h} \cap Q \cap \Omega) \leq C (\mathcal{N}_h^{*}(Q) +1) s(Q)^{k-1}
    \end{equation}
    where $s(Q)$ is the side length of $Q$.
\end{theorem} 
\begin{proof}
 Before we begin the proof, note that we are using the definition of the doubling index $\mathcal{N}_h^{*}$ provided in Definition \ref{DefVer2} and Definition \ref{CubeDoublingIndex}. In Lemma \ref{CubeDecomposition}, we constructed a decomposition for $\bar{\Omega}$ into two sets of cubes with disjoint interiors, namely $\bar{\Omega} \subset \bigcup_{i} Q_i^{int}  \cup \bigcup_{j} Q_j^{bdry} $, where $Q_i^{int}$ are  cubes entirely in the interior of $\Omega$, and $Q_j^{bdry}$ are cubes that are centered at the boundary containing a part of the boundary of $\Omega$ as well. The goal is now to prove estimate \eqref{ImportantTheoremIneqality} for all $Q_i^{int}$ and $Q_j^{bdry}$. The former case (when the cube is completely inside the domain) is an important result of \cite{MR1039348}. Having the desired result for the interior cubes, it suffices to prove \eqref{ImportantTheoremIneqality} assuming $Q$ is a boundary partition cube centered on $\partial\Omega$ of side $s$, with $s$ small enough depending on $\Omega$. Take $\widetilde{h}$ to be the reflection of $h$ in $\Omega \cup Q$ and suppose div$(B(x) \grad \widetilde{h}) = 0$ in our extended domain. From \cite{MR3739231}, we have the following polynomial estimate. There exists some $\alpha > 0$ such that 
   \begin{equation*}
        \mathcal{H}^{k-1}(Z_{\widetilde{h}} \cap \widetilde{Q}) \leq C (\mathcal{N}_{\widetilde{h}}(Q) +1)^\alpha s(Q)^{k-1}.
    \end{equation*}
Therefore,
   \begin{equation}\label{PolynomialEstimate}
        \mathcal{H}^{k-1}(Z_{h} \cap Q \cap \Omega) \leq C (\mathcal{N}_{\widetilde{h}}(Q) +1)^\alpha s(Q)^{k-1}.
    \end{equation}
Now, we will cover $Q\cap\Omega$ with smaller cubes of side length $2^{-M}s$, where $M = M(\Omega)$ comes from Lemma \ref{SmallCubeHalf}, in a specific way. First, we cover $Q\cap\partial\Omega$ by cubes $q$ centered at points of $\partial\Omega$ which have disjoint interiors. Then, we cover the rest of $Q \cap \Omega$ by interior cubes $q'$. Note that in this construction, the boundary cubes must be disjoint but the interior cubes can overlap. Let us call the former set of cubes $I$ for \textit{interior cubes} and the latter set $B$ for \textit{boundary cubes}. Set $\mathcal{N} := \mathcal{N}_h^{*}(Q)$. Now, since $q' \in I$, Lemma \ref{AlmostMonotonicity} gives us 
\begin{equation*}
    \mathcal{H}^{k-1}(Z_{h} \cap (\cup_{q'\in I} q')) \leq C(M,\Omega) \mathcal{N} s^{k-1}
\end{equation*}
for some constant $C(M,\Omega) > 0$. By Lemma \ref{SmallCubeHalf}, there exists a small boundary cube $q \in B$ such that $\mathcal{N}_h(q_0) < \frac{\mathcal{N}_h^{*}(Q)}{2} = \frac{\mathcal{N}}{2}$. Putting all of these results together, we get
\begin{equation*} \label{Inequality2}
\begin{split}
     \mathcal{H}^{k-1}(Z_{h} \cap Q \cap \Omega) &\leq \mathcal{H}^{k-1}(Z_{h} \cap (\cup_{q'\in I} q')) + \mathcal{H}^{k-1}(Z_{h} \cap q_0 \cap \Omega) +
    \sum_{q \in B, q \neq q_0} \mathcal{H}^{k-1}(Z_{h} \cap q \cap \Omega)\\
    &\leq C(M,\Omega) \mathcal{N} s^{k-1} + \mathcal{H}^{k-1}(Z_{h} \cap q_0 \cap \Omega) +
    \sum_{q \in B, q \neq q_0} \mathcal{H}^{k-1}(Z_{h} \cap q \cap \Omega)
\end{split}
\end{equation*}
Divide both sides by $s^{k-1}$ and define the following quantity 
\begin{equation*}
    \mu(\mathcal{N}) = \sup\frac{\mathcal{H}^{k-1}(Z_{h} \cap q)}{s(q)^{k-1}} 
\end{equation*}
where the $\sup$ is taken over all harmonic functions $h$ in $2Q\cap\Omega$ satisfying the Neumann boundary condition on $2Q \cap \partial\Omega$ such that $\mathcal{N}_h^{*}(Q) \leq \mathcal{N}$ and all cubes $q$ inside of $Q$. By \eqref{PolynomialEstimate}, we have $\mu(\mathcal{N}) < +\infty$. Now, taking  $\sup$ from both sides of the last inequality above, we get 
\begin{equation*}
    \mu(\mathcal{\mathcal{N}}) \leq C(M,\Omega)\mathcal{N} + 2^{-M(k-1)}\mu(\frac{\mathcal{N}}{2}) + (2^{M(k-1)}-1 )2^{-M(k-1)}\mu(\mathcal{N}).
\end{equation*}
Therefore, 
\begin{equation*}
    \mu(\mathcal{N}) \leq C(M,\Omega)\mathcal{N} + \mu(\frac{\mathcal{N}}{2}).
\end{equation*}
Iterating the above inequality gives us the desired result  
\end{proof}

\large{\textbf{Proof of Theorem 1.}} 
\normalsize Suppose $u_\lambda$ is a Neumann Laplace eigenfunction in a bounded domain $\Omega_0 \subset \mathbb{R}^n$ satisfying the conditions of Theorem \ref{MainTheorem}. We consider the harmonic extension $h(x,t)$ of $u_\lambda$ the way we defined it in Section $2$ in the bounded domain $\Omega = \Omega_0 \times [-1,1] \subset \mathbb{R}^{n+1}$. Take the decomposition $\bar{\Omega} \subset \bigcup_j Q_j^{bdry} \cup \bigcup_i Q_i^{int}$ into cubes $ \{Q_i^{int}\}$ and $\{Q_j^{bdry}\}$ as fixed in Lemma \ref{CubeDecomposition} but with a small change. In the aforementioned partition of our new cylinder-shaped domain $\Omega$, we are not considering the cubes that intersect the top and the bottom part of the cylinder as boundary cubes. In other words, we take $\{Q_j^{bdry}\}$ minus these specific cubes to be our boundary cubes, calling this subset $\{Q_{j_k}^{bdry}\}$.  Now, applying Theorem \ref{Theorem2} and Lemma \ref{sqrtBoundforN} to each of the interior and boundary cubes $Q_i^{int}$ and $Q_{j_k}^{bdry}$ respectively, we get 
\begin{equation*}
\begin{split}
    \mathcal{H}^{n}(Z_{h} \cap \Omega_1) &\leq \sum_i \mathcal{H}^{n}(Z_{h} \cap Q_i^{int}) + \sum_j \mathcal{H}^{n}(Z_{h} \cap Q_{j_k}^{bdry})\\
    &\leq C_1(C_0\sqrt{\lambda} + 1) \bigg(\sum_i s(Q_i^{int})^n + \sum_j s(Q_{j_k}^{bdry})^n\bigg) \leq C\sqrt{\lambda}
\end{split}
\end{equation*}
\makebox[\textwidth][s]{The above inequality gives us $\mathcal{H}^{n-1}(Z_{u_\lambda} \cap \Omega_0) \leq C \sqrt{\lambda}$, which concludes the proof of Theorem \ref{MainTheorem}. } \qed
\clearpage

\begin{appendices}
\titleformat{\section}[display]
    {\normalfont\Large\bfseries}{\appendixname\enspace\thesection}{.5em}{}
\section{Reflection Across the Boundary}
Take $\Omega \subset \mathbb{R}^n$ to be our domain and suppose $u$ is harmonic on $\Omega$ satisfying the Neumann boundary condition on $\partial\Omega$. Suppose in a neighborhood of the point $X \in \partial \Omega$, $\partial \Omega$ is the graph of a function $\phi$. Then, the normal vector of $\partial \Omega$ at $X := (x, \phi(x)) \in \partial\Omega$ pointing outwards is 
\begin{equation}\label{NormalVector}
    \nu(x) = \bigg (\frac{\grad \phi(x)}{\sqrt{1 + |\grad \phi(x)|^2}}, - \frac{1}{\sqrt{1 + |\grad \phi(x)|^2}} \bigg).
\end{equation}

Without loss of generality, using the proper translation and rotation, we assume  $0 \in \partial \Omega, \phi(0) = 0$, and $\grad \phi(0) = 0$. Then, ${\nu}(0) = (0,1) = e_n$ is the n-th standard unit basis vector of $\mathbb{R}^n$. Now, take $F: \mathbb{R}^{n-1} \times \mathbb{R}_{+} \rightarrow \mathbb{R}^n$ to be
\begin{equation}\label{FlatteningMap}
       F(x,s) = (x,\phi(x)) - s\cdot \big(\eta_s*{\nu}\big)(x).
\end{equation}
Here, $(\eta_s)_{s>0}$ is a smooth family of mollifiers $\eta_s(x) := \frac{1}{s^{n-1}} \eta(\frac{x}{s}) \in C^\infty_c(\mathbb{R}^{n-1})$, where $\eta \in C^\infty_c(\mathbb{R}^{n-1})$ is defined by 
\begin{equation}\label{ApproxId}
    \eta(x) = \Bigg\{
    \begin{array}{ll}
        Ce^{(|x|^2-1)^{-1}} & |x| < 1 \\
        0 & |x| \geq 1
    \end{array}
\end{equation}

with constant $C > 0$ chosen such that $\int \eta(x) dx = 1$. Here, note that $\eta_s*{\nu}$ means taking the convolution of $\eta_s$ with the components of the vector ${\nu}(x) \in \mathbb{R}^n$. Let us also define the extension map $\widetilde{F}: \mathbb{R}^{n-1} \times \mathbb{R} \rightarrow \mathbb{R}^n$ of $F$ as follows
\begin{equation}
    \widetilde{F}(x,s) = \Bigg\{
    \begin{array}{lll}
       F(x,s) & s \geq 0 \\
       (x,\phi(x)) & s = 0 \\
        G(x,-s) & s < 0
    \end{array}
\end{equation}
where $G(x,s) = (x,\phi(x)) + s\cdot \big(\eta_{s}*{\nu}\big)(x)$ for $ (x,s) \in \mathbb{R}^{n-1} \times \mathbb{R}_{+}$.
Now, it follows that for any $k = 1, ..., n-1$, 
\begin{equation}\label{kDerivative}
        \partial_kF(x,s) = (e_k, \partial_k \phi(x)) - 
    s \cdot \partial_k\big(\eta_s * \nu \big)(x) = (e_k, \partial_k \phi(x)) -
  \big((\partial_k\eta)_s * \nu\big)(x),
\end{equation}
and 
\begin{equation}\label{sDerivative}
           \partial_sF(x,s) = - \eta_s * \nu(x) -
    s \cdot \partial_s(\eta_s * \nu)(x) = - \eta_s * \nu(x) + 
 \bigg( \big((n-1) \eta(z) + z \cdot \grad \eta(z)\big)_s * \nu\bigg)(x).
\end{equation}
Also,
\begin{equation}\label{kDerivativeG}
        \partial_kG(x,-s) = (e_k, \partial_k \phi(x)) - 
    s \cdot \partial_k\big(\eta_{-s} * \nu \big)(x) = (e_k, \partial_k \phi(x)) +
  \big((\partial_k\eta)_{-s} * \nu\big)(x),
\end{equation}
and
\begin{equation}\label{sDerivativeG}
           \partial_sG(x,-s) = - \eta_{-s} * \nu(x) -
    s \cdot \partial_s(\eta_{-s} * \nu)(x) = - \eta_{-s} * \nu(x) + 
 \bigg( \big((n-1) \eta(z) + z \cdot \grad \eta(z)\big)_{-s} * \nu\bigg)(x).
\end{equation}
As an example, we will do the calculation for \eqref{sDerivative} as follows
\begin{equation*}\label{sDerivativeCalculation}
\begin{split}
           \partial_sF(x,s) &= - \eta_s * \nu(x) - 
    s \cdot \partial_s \bigg ( \int \frac{1}{s^{n-1}}\eta(\frac{z}{s}) \nu(x-z) dz) \bigg)\\
 &= - \eta_s * \nu(x) +
    (n-1) \int \frac{1}{s^{n-1}}\eta(\frac{z}{s}) \nu(x-z) dz + \int \frac{1}{s^{n-1}} \cdot \frac{z}{s}\grad\eta(\frac{z}{s}) \nu(x-z) dz \\
     &= - \eta_s * \nu(x) + 
 \bigg( \big((n-1) \eta(z) + z \cdot \grad \eta(z)\big)_s * \nu\bigg)(x).
    \end{split}
\end{equation*}
\begin{lemma}\label{DFLipschitz}
 If $\Omega$ has $\mathcal{C}^{1,1}$ boundary, $DF(x,s)$ is Lipschitz. \end{lemma}
 \begin{proof} Since $\Omega$ has $\mathcal{C}^{1,1}$ boundary, we have $\phi \in \mathcal{C}^{1,1}$, which also means $\nu(x)$ is Lipschitz. To show $DF$ is Lipschitz, we find $L > 0$ for which $||DF(x_1,s_1) - DF(x_2,s_2)|| \leq L |(x_1-x_2,s_1-s_2)|$, for any $x_1, x_2 \in \mathbb{R}^{n-1}$ and $s_1,s_2 \in \mathbb{R}_{+} $. Here, we use the norm $||.||_2$ for the matrices $DF(x_1,s_1)$ and $DF(x_2,s_2)$. Now, we know that for a matrix $A = (a_{ij})$, we have 
\begin{equation*}
    ||A||_2 \leq ||A||_F = \big (\sum_{i = 1}^{n} \sum_{j = 1}^{n} a_{ij}^2 \big)^\frac{1}{2} = \big ( \sum_{j = 1}^{n} |A_{j}|^2 \big)^\frac{1}{2}
\end{equation*}
where $\{A_j\}$ are the columns of $A$. Therefore,
\begin{equation*}
    \begin{split}
        ||DF(x_1,s_1) - DF(x_2,s_2)|| \leq 
        \big ( \sum_{k = 1}^{n} |\partial_kF(x_1,s_1) - \partial_kF(x_2,s_2)|^2 + |\partial_sF(x_1,s_1) - \partial_sF(x_2,s_2)|^2\big)^\frac{1}{2}.
    \end{split}
\end{equation*}
Thus, if we show that for $k = 1,..,n-1$, $\partial_kF$ and $\partial_sF$ are Lipschitz, then we are done. Since $\phi \in \mathcal{C}^{1,1}$, we get that $\partial_k\phi$ is Lipschitz. Note that $\eta_s * \nu$ is the mollification of $\nu$, so it is actually smooth, which gives us the local Lipschitz property. Since $\eta_s$ is compactly supported, $\eta_s * \nu$ is Lipschitz. Looking at the first lines of the equations \eqref{kDerivative} and \eqref{sDerivative}, we get the result. 
\end{proof}

As $s \rightarrow 0^{+}$, 
\begin{equation}\label{LimitofDerivativesofF}
    \partial_kF(x,s) \rightarrow (e_k, \partial_k\phi(x)), \quad \partial_sF(x,s) \rightarrow -\nu(x).
\end{equation}
Let us justify equations \eqref{LimitofDerivativesofF} before proceeding to the next part. Using \eqref{kDerivative}, we have 
\begin{equation*}
    \begin{split}
        \lim_{s \rightarrow 0^{+}} (\partial_k\eta)_s * \nu  =  \lim_{s \rightarrow 0^{+}} s \cdot \partial_k(\eta_s*\nu) 
        = \lim_{s \rightarrow 0^{+}} s \cdot \partial_k(\nu) = 0.
    \end{split}
\end{equation*}
Moreover, 
\begin{equation*}
        \lim_{s \rightarrow 0^{+}} -\eta_s * \nu  = -\nu.
\end{equation*}
Therefore, we only need to show that 
$ \lim_{s \rightarrow 0^{+}} \big((n-1) \eta(z) + z \cdot \grad \eta(z)\big)_s * \nu = 0$ to conclude \eqref{LimitofDerivativesofF}. Note that
\begin{equation*}
    \begin{split}
      \lim_{s \rightarrow 0^{+}}
\bigg( \big((n-1) \eta(z) + z \cdot \grad \eta(z)\big)_s * \nu\bigg)(x) &= \small{\lim_{s \rightarrow 0^{+}}  \frac{1}{s^{n-1}} \int_{\mathbb{R}^{n-1}} \big((n-1)s^{n-1}\eta(\tau) + s^{n-1}\tau\grad\eta(\tau) \big) \nu(x-s\tau) d\tau} \\
 &= \lim_{s \rightarrow 0^{+}}\int_{\mathbb{R}^{n-1}} \bigg((n-1)\eta(\tau) + \tau\grad\eta(\tau) \bigg) \nu(x-s\tau) d\tau \\
  &=  \nu(x) \int_{\mathbb{R}^{n-1}} (n-1)\eta(\tau) + \tau\grad\eta(\tau) d\tau.
    \end{split}
\end{equation*}
We will show that $\int_{\mathbb{R}^{n-1}} (n-1)\eta(\tau) + \tau.\grad\eta(\tau) d\tau = 0$. Since $\int_{\mathbb{R}^{n-1}} \eta(\tau) d\tau = 1$, we only need to show that $\int_{\mathbb{R}^{n-1}} \tau.\grad\eta(\tau) d\tau = -(n-1)$. This is clear since by using integration by parts,
\begin{equation*}
\begin{split}
      \int_{\mathbb{R}^{n-1}} \tau.\grad\eta(\tau) d\tau = \sum_{k=1}^{n-1} \int_{\mathbb{R}^{n-2}} \bigg (\int_{\mathbb{R}} \tau_k \partial_k\eta(\tau',\tau_k) d\tau_k \bigg) d\tau'  
    &=  \sum_{k=1}^{n-1} -\int_{\mathbb{R}^{n-2}} \int_{\mathbb{R}} \eta(\tau',\tau_k) d\tau_k d\tau'\\
    &= - \sum_{k=1}^{n-1} \int_{\mathbb{R}^{n-1}}  \eta(\tau) d\tau =  -(n-1).
\end{split}
\end{equation*}
Therefore, limits \eqref{LimitofDerivativesofF} occur. Now, as $s \rightarrow 0^{+}$, we have
\begin{equation}
    DF(x,s) \rightarrow 
    \begin{pmatrix}
    Id_{n-1} &  \bigg(- \frac{\grad \phi(x)}{\sqrt{1 + |\grad \phi(x)|^2}}\bigg)^{T} \\
    \grad \phi(x) & \frac{1}{\sqrt{1 + |\grad \phi(x)|^2}} 
    \end{pmatrix}.
\end{equation}
Thus, 
\begin{equation} \label{limitDF(0,s)}
    DF(0,s) \rightarrow 
    \begin{pmatrix}
    Id_{n-1} &  0\\
    0 & 1
    \end{pmatrix}
\end{equation}

is an invertible matrix for s sufficiently close to 0. Doing the exact same calculations for $G(x,s)$, since only there's a sign difference, we get 
\begin{equation} \label{limitDG(0,s)}
    \partial_kG(x,-s) \rightarrow (e_k, \partial_k\phi(x)), \quad \partial_sG(x,-s) \rightarrow -\nu(x)
\end{equation}
as $s \rightarrow 0^{-}$, and we get limit \eqref{limitDF(0,s)} for $DG(0,s)$ as well. Thus, $D\widetilde{F}$ is an invertible matrix for s sufficiently close to 0. Therefore, $\widetilde{F}$ is a local diffeomorphism near $(0,0)$, and there exists $\delta > 0$ such that $\widetilde{F}$ gives a parametrization of the $\delta$-neighborhood of the boundary in the normal directions. This means that in the region 
\begin{equation*}
    \mathcal{B}_\delta = \big\{ (x,s) \in \mathbb{R}^{n-1} \times \mathbb{R}: x \in B_\delta^{n-1}(0), s \in (-\delta, \delta) \big\}
\end{equation*}
the map $\widetilde{F}$ is a bijection to its image. Notice that the image $\widetilde{F}(\mathcal{B}_\delta)$ corresponds to the $\delta$-neighborhood of $ \big\{ (x,\phi(x)) \in \partial\Omega: x \in B_\delta^{n-1}(0) \big\}$, $\widetilde{F}(\mathcal{B}_\delta \cap (\mathbb{R}^{n-1} \times (\mathbb{R}_{+} \cup \{0\}))) \subset \Bar{\Omega}$, and $\widetilde{F}(\mathcal{B}_\delta \cap (\mathbb{R}^{n-1} \times \mathbb{R}_{-})) \subset \Omega^c$. We now prove the following general proposition before proceeding to construct an extension of $u$ satisfying the equation \eqref{NeumannLaplaceHarmonic} across the boundary. 

\begin{prop}\label{BiLipschitzCondition}
 Suppose $\psi: B \subset \mathbb{R}^n \rightarrow \mathbb{R}^n$ is such that $\psi: B \rightarrow \psi(B)$ is a bijection and $D\psi$ is bounded and away from 0 in $B$. Then $\psi^{-1}: \psi(B) \rightarrow B$ is also Lipschitz.
 \end{prop}
 \begin{proof}

The only thing that we need to show is that $D\psi^{-1}$ is bounded. Then, we know that a function with a bounded derivative is Lipschitz, and we get the result. Take $X \in \psi(B)$ and set $\psi(x_0) = X$. We have
\begin{equation*}
    ||D\psi^{-1}(X)|| = ||D\psi^{-1}(\psi(x_0))|| = ||D\psi(x_0)^{-1}||.
\end{equation*}
Now, $D\psi(x_0)$ is bounded, linear, and bijective, so by the inverse function theorem, $D\psi(x_0)^{-1}$ is also bounded and the result follows. 
\end{proof}
\subsubsection*{Flattening and Reflecting}
Taking $s\in\mathbb{R}_{+}$, we get $F(x,s) \in \Omega \subset \mathbb{R}^n$. First, we ``flatten" the boundary by defining the function $v: \mathbb{R}^{n-1} \times \mathbb{R}_{+} \rightarrow \mathbb{R}$ as
\begin{equation}
    v(x, s) = u (F(x,s)).
\end{equation}
 By Proposition \ref{BiLipschitzCondition}, F is bi-Lipschitz. Now, $v$ satisfies the elliptic equation div$(A(x,s) \grad v) = 0$  
 , where 
\begin{equation}\label{MatrixA}
\begin{split}
   A(x,s) &= |\det DF^{-1}|^{-1} \big( DF^{-1}(F(x,s)) \big)^T DF^{-1}(F(x,s))\\
   &= |\det DF(x,s)| \big( DF(x,s)^{-1} \big)^T DF(x,s)^{-1}.
   \end{split}
\end{equation}
\makebox[\textwidth][s]{Check chapter 11 of \cite{MR2145284} for the details of the calculation above. Moreover, $v$ satisfies} \makebox[\textwidth][s]{the Neumann boundary condition on the boundary of $\Lambda := \mathbb{R}^{n-1} \times \mathbb{R}_{+}$, which would be} $\partial \Lambda = \{ (x,s): x \in  \mathbb{R}^{n-1}, s = 0 \}$. We proceed to prove the following lemma before constructing a reflection of $v$.
 \begin{lemma}\label{ALipschitzCoeffs}
     The matrix $A(x,s)$ is symmetric and has Lispchitz coefficients. 
 \end{lemma}
 \begin{proof}
    From \eqref{MatrixA}, we know that $A(x,s) = |\det DF(x,s)| \big( DF(x,s)^{-1} \big)^T DF(x,s)^{-1}$, so the fact that it is symmetric clearly follows. To prove the other claim, it suffices to show that $DF(x,s)^{-1}$ has Lipschitz components. By Cramer's rule, 
     \begin{equation*}
       (DF)^{-1}  = (\det DF)^{-1} \text{adj}(DF) = (\det DF)^{-1}  \big((-1)^{i+j} M_{ij}\big)^T.
     \end{equation*}
Clearly, the components $M_{ij}$ are Lipschitz, and $\det DF$ is away from 0 and close to 1. Thus, the result follows.
\end{proof}

Now, let us look at the matrix $A(x,s)$ as $s$ goes to 0. From equations \eqref{LimitofDerivativesofF} and \eqref{MatrixA}, we get that as $s \rightarrow 0^{+}$, 
\begin{equation}\label{LimitofA}
    A(x,s) \rightarrow a(x)
    \begin{pmatrix}
    Id_{n-1}+\big(\grad\phi(x)\big)^T\grad\phi(x) & \quad 0 \\
    0 & \quad 1 
    \end{pmatrix},
\end{equation}
where $a(x)$ is the following non-negative scalar function
\begin{equation*}
    a(x) = \lim_{s \rightarrow 0^{+}} |\det DF(x,s)| = \frac{1}{\sqrt{1 + |\grad \phi(x)|^2}} \bigg|\sum_{k=1}^{n-1} (-1)^k\big(\partial_k\phi(x)\big)^2\bigg|.
\end{equation*}
Notice that the symmetric matrix in equation \eqref{LimitofA} has zeros on the off-diagonal block matrix, thus, since we also have the Neumann boundary condition, it is continuous with respect to an ``even" reflection of $v$. Precisely, we define an extension of $u$ across the boundary, denoted as $\widetilde{v}$ as 
\begin{equation}\label{Reflectedv}
    \widetilde{v}(x,s) = \Bigg\{
    \begin{array}{ll}
        v(x,s) & (x,s) \in \Bar{\Lambda}, \\
       v(x,-s)& (x,s) \in \Bar{\Lambda}^c.
    \end{array}
\end{equation}
Set $R(x,s):= (x,-s)$ to be the standard reflection. Again from \cite{MR2145284}, we have that outside of $\Bar{\Lambda}$, $\widetilde{v}$ satisfies an elliptic equation of the form div$(\widetilde{A}(x,s) \grad \widetilde{v}) = 0$, where 
\begin{equation}\label{MatrixAtilde}
    \widetilde{A}(x,s) = |\det DR(x,s)| \big( DR(x,s)^{-1} \big)^T A\big(DR(x,s)^{-1}\big) 
    = \big( DR(x,s) \big)^T A\big(DR(x,s)\big).
\end{equation}
Clearly, the components $\widetilde{a}_{ij}$ of $\widetilde{A}$ are defined for $s < 0$ as follows
\begin{equation*}
   \begin{split}
       \widetilde{a}_{in} (x,s) = -a_{in}(x,-s) & \qquad i \neq n\\
        \widetilde{a}_{nj} (x,s) = -a_{nj}(x,-s) & \qquad j \neq n\\
       \end{split}
\end{equation*}
and 
\begin{equation*}
     \widetilde{a}_{ij} (x,s) = a_{ij}(x,-s) 
\end{equation*}
in the other cases. By equation \eqref{LimitofA}, we can extend $\widetilde{A}$ to be equal to $A$ for $s \geq 0$, so we would have div$(\widetilde{A}(x,s) \grad \widetilde{v}) = 0$ for $\widetilde{v}$ everywhere. Note that by Lemma \ref{ALipschitzCoeffs} and the way $\widetilde{A}$ is defined, the matrix $\widetilde{A}$ also has Lipschitz coefficients. Next, we define an extension $\widetilde{u}$ of $u$ across the boundary of $\Omega$ for $X \in \widetilde{F}(\mathcal{B}_\delta)$ in a ball centered on the boundary as
 \begin{equation}\label{Reflectedu}
     \widetilde{u}(X) = \widetilde{v}(\widetilde{F}^{-1}(X)).
 \end{equation}

Note that due to Proposition \ref{BiLipschitzCondition}, $\widetilde{F}$ is bi-Lipschitz. By the calculations done in chapter 11 of \cite{MR2145284}, $\widetilde{u}$ satisfies the elliptic equation div$(B(X) \grad \widetilde{u}) = 0$, where 
\begin{equation}\label{MatrixB}
   B(X) = |\det D\widetilde{F}|^{-1} \big( D\widetilde{F}\big)^T \widetilde{A}\big(D\widetilde{F}\big).
\end{equation}
The following lemma is one of the main results of this Section, which can be proved in the exact same manner as Lemma \ref{ALipschitzCoeffs}. We state it below since it will be referred to in other Sections multiple times. 
\begin{lemma} \label{BLipschitzCoeffs}
     The matrix $B(X)$ is symmetric and has Lispchitz coefficients. \qed
\end{lemma} 

We end the appendix by reformulating $\widetilde{u}$ in terms of $u$, which leads to an estimate for the way our reflection moves points.
\begin{lemma} \label{Phi}
   Take $\widetilde{u}$ to be the reflection to $u$ defined as in equation \eqref{Reflectedu} in a neighborhood $U$ of a point $X_0 \in \partial \Omega$. Then, there exists a Lipschitz function $\Phi$ such that $\widetilde{u}(X) = u(\Phi(X))$ for every $X \in U$.
\end{lemma}
\begin{proof}
    Set $\Phi(X):= F \circ R \circ \widetilde{F}^{-1}(X) $, where $R(x,s)$ is defined in \eqref{Reflectedv}. This function $\Phi$ is Lipschitz since it is a composition of Lipschitz functions. Moreover, due to the construction method of the reflection $\widetilde{u}$, we have $\widetilde{u}(X) = u(\Phi(X))$.
\end{proof}

\begin{corollary} \label{DistanceEstimatePhi}
  Fixing the reflection in the neighborhood $U$ of the point $X_0$ as in the statement of Lemma $\ref{Phi}$, there exists $c > 0$ depending only on $\Omega$ such that for every $X \in U$,
  \begin{equation*}
      |\Phi(X) - X_0| \leq c |X - X_0|
  \end{equation*}
\end{corollary}
\begin{proof}
    We have $\Phi(X_0) = X_0$, thus, taking $c$ to be the Lispchiz constant of $\Phi$, we have
    \begin{equation*}
           |\Phi(X) - X_0| = |\Phi(X) - \Phi(X_0)| \leq c |X - X_0|.
    \end{equation*}
\end{proof}
\clearpage

\section{\makebox[\textwidth][s]{Examples of the Nodal Sets of Neumann Laplace Eigenfunctions}}
 \subsection*{Disk}
To determine the eigenfunctions of the Neumann Laplace operator on a disk, we first use the rotation symmetry of such domain to write  the Laplace operator in polar coordinates as
\begin{equation*}
   \Delta = \frac{\partial^2}{\partial r^2} + \frac{1}{r} \frac{\partial}{\partial r} + \frac{1}{r^2} \frac{\partial^2}{\partial \theta^2},
\end{equation*}
where $r > 0$ and $\theta \in [0, 2 \pi)$ are the polar coordinates in $x = r \cos \theta$ and $y = r \sin \theta$. Using the separation of variables, we are able to determine the following explicit representation for the eigenfunctions of the Neumann Laplacian in a disk $D_R = \{ x \in \mathbb{R}^2 : 0 < |x| < R\}$   
\begin{equation*}
u_{nmk}(r, \theta) = 
\Bigg\{
    \begin{array}{ll}
          J_n(\widetilde{j}_{nm} r/R) \cos(n\theta) & \text{\quad} k = 1, \\
        J_n(\widetilde{j}_{nm} r/R) \sin(n\theta) & \text{\quad} k = 2 \quad (n \neq 0).
    \end{array}
\end{equation*}
Here, The eigenfunctions are
indexed by $nmk$, with $n = 0, 1, 2,...$ corresponding to  the order of Bessel
functions, $m = 1, 2, 3,...$ counting solutions of the equation we get from the separation of the variables, and $k = 1, 2$. Since $u_{0m2}(r, \theta)$
are zero (as $\sin(n\theta) = 0$ for $n = 0$), they are excluded. In addition, $\widetilde{j}_{nm}$ are the positive roots of the derivative $J_n'(z)$ of the Bessel function $J_n(z)$. The eigenvalues are equal to $\lambda_{nm} = \frac{\widetilde{j}_{nm}^2}{R^2}$. We refer the reader to \cite{MR3124880}, Section 3.2 for more details. 
Let us now investigate the nodal sets of the eigenfunctions for a simple case of the above formulations. Suppose $n = 0$, so we only have to calculate the zero sets of $u_{0m1} =  J_0(\widetilde{j}_{0m} r/R)$. 
\begin{equation*}
   Z_{u_{0m1}} = \big\{(r,\theta) : 0 < r < R, \quad 0 \leq \theta < 2\pi,  \quad J_0(\widetilde{j}_{0m} r/R) = 0 \big\}.
\end{equation*}
Before we continue, just note that in this special case when $n = 0$, we have the following identity
\begin{equation*}
   {J}'_0(z)=\sum_{k=1}^\infty \frac {\left(-1\right)^k \,k\;\left(\frac z2\right)^{2k-1}}{k!\,k!}=-\sum_{j=0}^\infty \frac {\left(-1\right)^j \;\left(\frac z2\right)^{2j+1}}{\Gamma(j+2)j!}=-{J}_1(z).
\end{equation*}
Thus, each positive root of the derivative $J_0'(z)$ of the Bessel function $J_0(z)$ is also a positive root of the Bessel function ${J}_1(z)$ and vice versa, thus, we have $\widetilde{j}_{0m} = j_{1m}$. From now on, we will use $j_{1m}$ instead of $\widetilde{j}_{0m}$. 
Now, take $\{j_{0\ell}\}$ to be the set of positive roots of the Bessel function $J_0(z)$ in increasing order. The condition $J_0(\widetilde{j}_{0m} r/R) = J_0(j_{1m}r/R) = 0$ above leads to the existence of some $\ell$ such that $j_{1m} r/R = j_{0\ell}$, which gives us the inequality $0 < r = Rj_{0\ell}/j_{1m} < R $. Thus, for any $\ell$ for which we have $0 < j_{0\ell}/j_{1m} < 1 $, we will get a radius $r_\ell = Rj_{0\ell}/\widetilde{j}_{0m}$ such that $u_{0m1}(r_\ell,\theta) = 0$ for any $\theta$. For every $m \in \mathbb{N}$, take $N_m$ to  be the largest positive integer such that for any $1 \leq \ell \leq N_m$, we have $0 < j_{0\ell}/j_{1m} < 1 $, which gives us $ 0 < r_\ell < R$ and $J_0(j_{1m} r_\ell/R) = 0$. We claim that $N_m = m$, which follows straightforwardly from the fact that the zeros of Bessel functions with shifted parameters interlace. The following proposition states this property; see Section 15.22 of \cite{watson1922treatise}.
\begin{prop} 
    For any $\nu$, the positive real zeros of the Bessel functions $J_\nu(x)$ and $J_{\nu+1}(x)$ interlace according to the following inequality.
   \begin{equation*}
       0 < j_{\nu,1} < j_{\nu+1,1} < j_{\nu,2} < j_{\nu+1,2} < j_{\nu,3} < \ldots
   \end{equation*}
\end{prop}
Using the above proposition we see that $N_m = m$ is the largest positive integer for which we have $0 < j_{0N_m}/j_{1m} < 1 $. Therefore, summing up all the results so far,  we expect the nodal sets of $u_{0m1}$ to be a union of $m$ circles centered at the origin of radii $r_\ell = Rj_{0\ell}/j_{1m}$, where $ 1 \leq \ell \leq m$. The Hausdorff measure of this set is of the form
\begin{equation} \label{H1ExpansionDisk}
\begin{split}
     \mathcal{H}^1(Z_{u_{0m1}}) = \sum_{\ell = 1}^{m} 2\pi r_\ell &= 2\pi(r_1 + ... r_{m})
     = 2 \pi \sum_{\ell = 1}^{m} \frac{R j_{0\ell}}{\widetilde{j}_{0m}} 
     = {2\pi R} \sum_{\ell = 1}^{m} \frac{j_{0\ell}}{{j}_{1m}}
\end{split}
\end{equation}
For $\nu \geq 0$, we have the following lower bound for $j_{\nu m}$ as the $m-$th positive root of the Bessel function $J_\nu$ in Section 15.3 of \cite{watson1922treatise}. 
\begin{equation} \label{LowerBoundjm}
    \left(\left(m - \frac{1}{4}\right)^2 + \nu ^2\right)^{\frac{1}{2}} \leq j_{\nu m}
\end{equation}
Now, we need to find $C>0$ only depending on the disk $D_R$ such that $ \mathcal{H}^1(Z_{u_{0m1}}) \leq C\sqrt{\lambda_{0m}}$ for any $m$. By \eqref{H1ExpansionDisk}, we have 
\begin{equation}\label{FirstInequalityH1}
     \mathcal{H}^1(Z_{u_{0m1}}) = {2\pi R} \sum_{\ell = 1}^{m} \frac{j_{0\ell}}{{j}_{1m}} \leq {2\pi R} \sum_{\ell = 1}^{m} 1 = {2\pi R}m
\end{equation}
Thus, if we find $C = C(D_R)$ described above such that $ {2\pi R}m \leq C\sqrt{\lambda_{0m}} = C \frac{j_{1m}}{R}$, or equivalently $\frac{2\pi R^2m}{j_{1m}} \leq C$, we will have the desired result. In this particular example, one such constant that works is $C = 2\sqrt{\frac{17}{16}}\pi R^2$. To show that this constant $C$  satisfies our desired inequality, we use the estimate  \eqref{LowerBoundjm} for $\nu =1$ and a simple Cauchy-Schwarz inequality as follows.
\begin{equation*}
    \begin{split}
        C j_{1m} \geq  (2\pi R^2) \left(\left(\frac{17}{16}\right) \left(\left(m - \frac{1}{4}\right)^2 + 1\right)\right)^{\frac{1}{2}} & = (2\pi R^2) \left(\left(1 + \frac{1}{4^2}\right) \left(\left(m - \frac{1}{4}\right)^2 + 1\right)\right)^{\frac{1}{2}} \\
        &\geq (2\pi R^2) \left(\left(m - \frac{1}{4}\right) + \frac{1}{4}\right) = {2\pi R^2}m
    \end{split}
\end{equation*}
which gives us  $ {2\pi R}m \leq C \frac{j_{1m}}{R}$ as desired. Note that in this example, it turns out that the constant  $C$ depends on $\pi R^2$, which is the area of the disk $D_R$.
\subsection*{Rectangle}
Let us calculate the surface area of the nodal set of the Neumann Laplace eigenfunctions on a rectangle. Take $\Omega = [0,a] \times [0,b] \subset \mathbb{R}^2$ to be a rectangle for $a, b > 0$. Let $u_\lambda$ be a solution to equation \ref{NeumannLaplaceEigenfunction} in $\Omega$. In a rectangle, the eigenvalues of the Neumann Laplace operator are of the form $\lambda_{nm} = \frac{\pi^2 n^2}{a^2} + \frac{\pi^2 m^2}{b^2} $ for any $n, m \in \mathbb{N}$; see \cite{MR3124880}, Section 3.1. Moreover, the eigenfunction associated with $\lambda_{nm}$, which we denote by $u_{nm}$ is of the form 
$u_{nm}(x,y) = \cos(\frac{\pi n x}{a})\cos(\frac{\pi m y}{b})$ for any $(x,y) \in \Omega = [0,a] \times [0,b] $. Now, to calculate the nodal set $Z_{u_{nm}} = \{u_{nm} = 0\}$, we see that 
\begin{equation*}
   Z_{u_{nm}} = \big\{(x,y) \in [0,a] \times [0,b]: \cos(\frac{\pi n x}{a}) = 0\big\} \bigcup \big \{ (x,y) \in [0,a] \times [0,b]: \cos(\frac{\pi m y}{b}) = 0\big\}.
\end{equation*}
In order to have $\cos(\frac{\pi n x}{a}) = 0$, we should have $0 < x = \frac{(2k+1)a}{2n} < a$ for an integer $k$, which leaves us with $ 0 < 2k+1 < 2n$. This gives us $ 0 \leq k \leq n-1$. Similarly, in order to have $\cos(\frac{\pi m y}{b})= 0$, we should have $0 < y = \frac{(2\ell+1)b}{2m} < b$ for an integer $\ell$, which leaves us with $ 0 < 2\ell+1 < 2m$. This gives us $ 0 \leq \ell \leq m-1$. Therefore, our nodal set $Z_{u_{nm}}$ consists of two sets of lines, which are $n$ vertical lines of length $b$ and $m$ horizontal lines of length $a$. Thus, we have
\begin{equation*}
  \mathcal{H}^1(Z_{u_{nm}}) = n b + m a.
\end{equation*}
Now, set $C := \frac{ab\sqrt{2}}{\pi}$. By a simple Cauchy-Schwarz inequality, we see that 
\begin{equation*}
\begin{split}
    C\sqrt{\lambda_{nm}}  = \frac{ab\sqrt{2}}{\pi} \sqrt{\frac{\pi^2 n^2}{a^2} + \frac{\pi^2 m^2}{b^2}} 
    & = \sqrt{(\frac{a^2b^2}{\pi^2} + \frac{a^2b^2}{\pi^2})(\frac{\pi^2 n^2}{a^2} + \frac{\pi^2 m^2}{b^2})}\\
     & \geq \sqrt{n^2 b^2} + \sqrt{m^2 a^2} = 
    n b + m a = \mathcal{H}^1(Z_{u_{nm}})
\end{split}
\end{equation*}
as we calculated. Note that this upper bound is sharp since as $\frac{n}{m}$ approaches $\frac{a}{b}$, the equality condition for the Cauchy-Schwarz holds for this constant $C$. In addition, note that the constant $C$ here depends on $ab$, which is the area of the rectangle.
\end{appendices}

\bibliographystyle{plain}
\bibliography{Neumann_Laplace_Eigenfunctions}

\begin{thebibliography}{10}

\bibitem{Alessandrini_2009}
G.~Alessandrini, L.~Rondi, E.~Rosset, and S.~Vessella.
\newblock The stability for the {C}auchy problem for elliptic equations.
\newblock {\em Inverse Problems}, 25(12):123004, nov 2009.

\bibitem{MR3396451}
K.~Bellov\'{a} and F.-H Lin.
\newblock Nodal sets of {S}teklov eigenfunctions.
\newblock {\em Calc. Var. Partial Differential Equations}, 54(2):2239--2268,
  2015.

\bibitem{MR4577966}
N.~Burq and I.~Moyano.
\newblock Propagation of smallness and control for heat equations.
\newblock {\em J. Eur. Math. Soc. (JEMS)}, 25(4):1349--1377, 2023.

\bibitem{MR2145284}
L.~Caffarelli and S.~Salsa.
\newblock {\em A geometric approach to free boundary problems}, volume~68 of
  {\em Graduate Studies in Mathematics}.
\newblock American Mathematical Society, Providence, RI, 2005.

\bibitem{decio2022hausdorff}
S.~Decio.
\newblock Hausdorff measure bounds for nodal sets of {S}teklov eigenfunctions,
  ar{X}iv: 2104.10275, 2022.

\bibitem{MR0943927}
H.~Donnelly and C.~Fefferman.
\newblock Nodal sets of eigenfunctions on {R}iemannian manifolds.
\newblock {\em Invent. Math.}, 93(1):161--183, 1988.

\bibitem{MR1039348}
H.~Donnelly and C.~Fefferman.
\newblock Nodal sets of eigenfunctions on {R}iemannian manifolds with boundary.
\newblock In {\em Analysis, et cetera}, pages 251--262. Academic Press, Boston,
  MA, 1990.

\bibitem{MR0833393}
N.~Garofalo and F.-H. Lin.
\newblock Monotonicity properties of variational integrals, {$A_p$} weights and
  unique continuation.
\newblock {\em Indiana Univ. Math. J.}, 35(2):245--268, 1986.

\bibitem{garofalo1987unique}
N.~Garofalo and F.-H. Lin.
\newblock Unique continuation for elliptic operators: a geometric variational
  approach.
\newblock {\em Communications on pure and applied mathematics}, 40(3):347--366,
  1987.

\bibitem{MR3124880}
D.~S. Grebenkov and B.-T. Nguyen.
\newblock Geometrical structure of {L}aplacian eigenfunctions.
\newblock {\em SIAM Rev.}, 55(4):601--667, 2013.

\bibitem{MR2251558}
A.~Henrot.
\newblock {\em Extremum problems for eigenvalues of elliptic operators}.
\newblock Frontiers in Mathematics. Birkh\"{a}user Verlag, Basel, 2006.

\bibitem{kenig2024note}
C.~Kenig and Z.~Zhao.
\newblock A note on the critical set of harmonic functions near the boundary,
  ar{X}iv: 2402.08881, 2024.

\bibitem{MR3739231}
A.~Logunov.
\newblock Nodal sets of {L}aplace eigenfunctions: polynomial upper estimates of
  the {H}ausdorff measure.
\newblock {\em Ann. of Math. (2)}, 187(1):221--239, 2018.

\bibitem{MR3739232}
A.~Logunov.
\newblock Nodal sets of {L}aplace eigenfunctions: proof of {N}adirashvili's
  conjecture and of the lower bound in {Y}au's conjecture.
\newblock {\em Ann. of Math. (2)}, 187(1):241--262, 2018.

\bibitem{MR4249624}
A.~Logunov and E.~Malinnikova.
\newblock Lecture notes on quantitative unique continuation for solutions of
  second order elliptic equations.
\newblock In {\em Harmonic analysis and applications}, volume~27 of {\em
  IAS/Park City Math. Ser.}, pages 1--33. Amer. Math. Soc., [Providence], RI,
  2020.

\bibitem{MR4356702}
A.~Logunov, E.~Malinnikova, N.~Nadirashvili, and F.~Nazarov.
\newblock The sharp upper bound for the area of the nodal sets of {D}irichlet
  {L}aplace eigenfunctions.
\newblock {\em Geom. Funct. Anal.}, 31(5):1219--1244, 2021.

\bibitem{watson1922treatise}
G.~Watson.
\newblock {\em A treatise on the theory of {B}essel functions}.
\newblock 2nd ed., Cambridge Math Lib., Cambridge University Press, Cambridge,
  UK, 1995.

\bibitem{Yau1982}
S.-T. Yau.
\newblock Problem section, seminar on differential geometry.
\newblock {\em Annals of Mathematical Studies 102, Princeton, 1982, 669–706},
  1982.

\bibitem{MR4594557}
J.~Zhu.
\newblock Boundary doubling inequality and nodal sets of {R}obin and {N}eumann
  eigenfunctions.
\newblock {\em Potential Anal.}, 59(1):375--407, 2023.

\bibitem{zhu2023nodal}
J.~Zhu and J.~Zhuge.
\newblock Nodal sets of {D}irichlet eigenfunctions in quasiconvex {L}ipschitz
  domains, ar{X}iv: 2303.02046, 2023.

\end{thebibliography}

\bigskip
(Shaghayegh Fazliani) \textsc{Department of Mathematics, Stanford University,
    450 Serra Mall, Building 380, Stanford, CA, 94305}\\
  \textit{Email address:} \texttt{fazliani@stanford.edu}

\end{document}